\documentclass[10pt]{article}

\title{Optimal Trade Characterizations in Multi-Asset Crypto-Financial Markets }
\author{Carlos Escudero\thanks{Departamento de Matem\'aticas Fundamentales,
Universidad Nacional de Educaci\'on a Distancia, Madrid, Spain.
E-mail: cescudero@mat.uned.es.}
\and
Felipe Lara\thanks{Instituto de Alta Investigaci\'on (IAI), Universidad
de Tarapac\'a, Arica, Chile. E-mail: felipelaraobreque@gmail.com;
flarao@academicos.uta.cl. Web: www.felipelara.cl, ORCID-ID: 0000-0002-9965-0921.}
\and
Miguel Sama\thanks{Departamento de Matem\'atica Aplicada I,
Universidad Nacional de Educaci\'on a Distancia, Madrid, Spain.
E-mail: msama@ind.uned.es.}
}

\date{\small \emph{\today}}
\usepackage{graphicx}
\graphicspath{{./Figures/}}
\usepackage{amsmath}
\usepackage{amsthm}
\usepackage{amssymb}
\usepackage{multicol}
\usepackage{subcaption}
\usepackage{hyperref}
\usepackage{array}
\usepackage[dvipsnames,svgnames]{xcolor}

\numberwithin{equation}{section}

\newcommand{\interior}{\mathop{\rm int}\nolimits}

\newcommand{\BE}{\begin{equation}}
\newcommand{\EE}{\end{equation}}
\newcommand{\BEQ}{\begin{eqnarray}}
\newcommand{\EEQ}{\end{eqnarray}}
\newcommand{\BD}{\begin{displaystyle}}
\newcommand{\ED}{\end{displaystyle}}
\newcommand{\BC}{\begin{center}}
\newcommand{\EC}{\end{center}}
\newtheorem{thr}{Theorem}[section]
\newtheorem{defn}{Definition}[section]
\newtheorem{pr}{Proposition}[section]
\newtheorem{co}{Corollary}[section]

\def\interior{\operatornamewithlimits{int}}
\def\cone{\operatornamewithlimits{cone}}

\def\cone{\operatornamewithlimits{cone}}

\newtheorem{corollary}{Corollary}[section]

\newtheorem{remark}{Remark}[section]

\begin{document}

\maketitle

\renewcommand{\thefootnote}{}

\begin{abstract}

\noindent {\bf Abstract.}
This work focuses on the mathematical study of constant function market makers.
We rigorously establish the conditions for optimal trading under the assumption of a
quasilinear, but not necessarily convex (or concave), trade function. This generalizes previous results that used convexity, and also guarantees the robustness against
arbitrage of so-designed automatic market makers. The theoretical results are
illustrated by families of examples given by generalized means, and also by numerical
simulations in certain concrete cases. These simulations along with the mathematical
analysis suggest that the quasilinear-trade-function based automatic market makers
might replicate the functioning of those based on convex functions, in particular
regarding their resilience to arbitrage.

\medskip


\end{abstract}

\section{Introduction}

Cryptocurrencies have had a deep impact in the XXIst century world. The Bitcoin white paper \cite{SN} has not only given rise to an asset whose market price
has grown to several dozens of thousands of dollars, but has also been cited tens of thousands of times. Its impact has then ranged from financial practice to academic research. The rise of cryptocurrencies has also generated a debate about how finance
can be evolved accordingly. This is the starting point
of the revolutionary decentralized finance (DeFi).

Automatic Market Makers (AMMs) have their own status within decen\-tra\-li\-zed
finance.
They were born to provide an automated liquidity provision, along with trading services,
on blockchain networks. Nowadays, some AMMs such as Uniswap \cite{AZS}, Balancer \cite{MM}, and Curve \cite{egorov19}
enjoy a huge popularity among practitioners. Their main functionality is to enable decentralized exchange (DEX) of digital assets. This, in turn, facilitates trading,
lending, and yield farming.

An AMM is a smart contract protocol that permits to trade cryptoassets without the need of an order book. They are based on exchange algorithms that,
together with liquidity pools, set the trading framework. Each algorithm determines how the corresponding AMM works, as it is responsible of tuning the
asset prices, providing liquidity, and the market efficiency. Its functioning, in consequence, differs sharply from that of traditional centralized
exchanges, which match the orders emitted by buyers and sellers.

Several advantages and disadvantages can be associated with AMMs. On the positive
side, they are assumed to facilitate the democratization
of financial markets. Any trader can execute exchanges according to what is set by the predetermined algorithm, or either deposit digital assets
in the liquidity pools for a fee. Therefore, AMMs reduce the need to rely on intermediaries and counterpart risk if compared to centralized exchanges.
Nevertheless, AMM users can also experience a number of negative aspects, such as impermanent loss, front-running, and slippage. The design of
AMM protocols, and in particular of the exchange algorithms, is therefore key to improve their features and to diminish their handicaps.

This work is devoted to the construction of a general mathematical framework for the theoretical design of exchange algorithms.
We depart from the model in~\cite{angeris2022constant} to consider the optimization problem:
\begin{equation}\label{for:intro}
 \begin{array}{lll}
  & \text{maximize} & U(x-y) \\
  & \text{subject to} & \varphi (R+\Gamma y-x)=\varphi (R)\text{, } \\
  & \, & R + \Gamma y-x \geq 0\text{, } \\
  & \, & R\geq x\geq 0\text{, }y\geq 0.
 \end{array}%
\end{equation}
Herein $U:\mathbb{R}^n \longrightarrow \mathbb{R}$ is the utility function that
models the preferences of the trader (and which, of course, the trader seeks to
maximize), $n$ is the number of cryptoassets, $x,y \in \mathbb{R}^n$ are the
baskets of assets demanded from and sent to the AMM, respectively, $R\in\mathbb{R}_{++}^{n}$ are the
reserves of the AMM,
and $\Gamma$ is an $n \times n$ diagonal matrix that encodes the fees charged by the AMM to the trader for performing the trade. The function
$\varphi:\mathbb{R}^n \longrightarrow \mathbb{R}$ is the one that controls the exchange algorithm, as it should be kept constant at every trade;
thus the name ``constant function market maker'' for this type of AMM
\cite{angeris2020} (see also \cite{H}). The rest of the constraints are there to
keep the non-negativity of the
reserves and of the quantities of assets traded; note that the vectorial inequalities should be understood componentwise. In the case of
Uniswap, for instance, the function $\varphi(\cdot)$ can be taken to be the geometric mean (or equivalently the product of the components).
This case along with several generalizations and variants have
been extensively studied in the literature, see for instance~\cite{angeris2020,angeris2022,angeris2023,angeris2021},
where this list is not meant to be exhaustive. We depart from these studies, and get inspired by them,
and more particularly of~\cite{angeris2022constant}, to carry out our analysis, which
is actually a generalization of some instances of the latter work.

\medskip

{\bf Our contribution:} We provide necessary and sufficient optimality conditions for
optimal trades in multi-asset crypto-financial markets beyond convexity, that is, when
both trading and utility functions satisfy generalized convexity assumptions. This improvement is more than theoretical since it has economical/financial motivations for the potential development of new crypto-financial markets based on AMM algorithms.
As it is well-known by the consumer pre\-fe\-ren\-ce theory \cite{D-1959}, utility
functions are naturally assumed quasiconcave because preferences that result in
concave utility functions are
often considered artificial (see, for instance, \cite{D-1959,MWG}). Furthermore, we
consider the trading function to be quasilinear, a generalized convexity assumption
which  includes the nonincreasing/nondecreasing convex (or concave) case,
preserves the convexity of problem \eqref{for:intro} and also the equality in the
constraint without being restricted to the simple linear (affine) case. Moreover, the
new optimality conditions that we establish are valid for any trade while,
until now, only the no-trade case (when the optimal pair is $(0,0)$, i.e. it is best not
to trade whatsoever) was known in the literature for the convex case \cite{angeris2022constant}. Our proofs do not just cover the case of a uniform fee
as assumed in \cite{angeris2022constant}, but also extend the analysis to the case
of asset-dependent fees,
which is a generalization that can be implemented in real AMMs. Finally, we apply our results to construct new trading functions that
are generalized means. This obviously includes as particular cases the geometric and arithmetic means, which has been used in both academic
articles and real AMMs. Nevertheless, our results include the case of quasilinear generalized means that are neither convex nor concave, which
were not regarded in previous theoretical developments, and open the possibility to
start the design of new AMMs built upon this new type
of trading functions.

\medskip

The paper is organized as follows.
In section \ref{sec:02} we introduce the definitions and mathematical terminology that
will be employed from there onwards. In section \ref{sec:03} we
formulate the problem of optimal trading in an AMM as a mathematical optimization, making precise the above description. In section \ref{sec:04}, the results of the previous section are
employed to characterize the optimality of not executing any trade, that is, to characterize when no trading whatsoever is the optimal choice. This is
important to establish the robustness of the AMM to the action of arbitrageurs, since under those conditions no trades will be executed just for the sake of profit. In section \ref{sec:05}, we illustrate our theoretical developments with
particular examples. We focus on the case in which the trade function is neither convex nor concave, something that escapes previous theoretical frameworks.
We make these ideas concrete by employing generalized means, which we propose as trade functions, and combine theoretical and numerical analyses
to assess them. Finally, in section \ref{sec:06}, we draw our conclusions.

\section{Preliminaries}\label{sec:02}

We denote $\mathbb{R}_{+} := [0, + \infty[$ and $\mathbb{R}_{++}
:= \, ]0, + \infty[$. Hence, $\mathbb{R}^{n}_{+} := [0, + \infty[ \, \times
\ldots \times [0, + \infty[$ and $\mathbb{R}^{n}_{++} := \, ]0, + \infty[
\times \ldots \times \,  ]0, + \infty[$ ($n$ times). We use the usual
notations $\geq$ componentwise (see \cite{ehrgott2005multicriteria})
\begin{align*}
 & ~~ x \geq y ~ \Longleftrightarrow ~ x_{i} \geq y_{i}, \\
 & ~~ x \gneq y ~ \Longleftrightarrow ~ x_{i} \geq y_{i}, x \neq y, \\
 & x >> y ~ \Longleftrightarrow ~ x_{i} > y_{i},
\end{align*}
for every $i=1,...,n$.

Let $K \subseteq \mathbb{R}^{n}$ be nonempty. Then the set $K^{\ast}
\subseteq \mathbb{R}^{n}$ is the polar (positive) cone of $K$ defined by
\begin{equation} \label{polar:cone}
 K^{\ast} := \{q \in \mathbb{R}^{n}: ~ \langle q, p \rangle \geq 0, ~
 \forall ~ p \in K\}.
\end{equation}

Given any extended-valued function $h: \mathbb{R}^{n} \rightarrow
\overline{\mathbb{R}} := \mathbb{R} \cup \{\pm \infty\}$, the effective
domain of $h$ is defined by ${\rm dom}\,h := \{x \in \mathbb{R}^{n}:
h(x) < + \infty \}$. It is said that $h$ is proper if ${\rm dom}\,h$ is
nonempty and $h(x) > - \infty$ for all $x \in \mathbb{R}^{n}$. The notion
of properness is important when dealing with minimization pro\-blems.

It is indicated by ${\rm epi}\,h := \{(x,t) \in \mathbb{R}^{n} \times
\mathbb{R}: h(x) \leq t\}$ the epigraph of $h$, by $S_{\lambda} (h) :=
\{x \in \mathbb{R}^{n}: h(x) \leq \lambda\}$ (resp. $S^{<}_{\lambda}
(h) := \{u \in \mathbb{R}^{n}: ~ h(u) < \lambda\}$ the sublevel (resp.
strict sublevel) set of $h$ at the height $\lambda \in \mathbb{R}$, by
 $W_{\lambda} (h) := \{x \in \mathbb{R}^{n}: h(x) \geq \lambda\}$ (resp. $W^{<}_{\lambda} (h) := \{u \in \mathbb{R}^{n}: ~ h(u) > \lambda\}$
the upper level (resp. strict upper level) set of $h$ at the height $\lambda
\in \mathbb{R}$, and by ${\rm argmin}_{\mathbb{R}^{n}} h$ the set
of all minimal points of $h$.

A function $h$ with convex domain is said to be
\begin{itemize}
 \item[$(a)$] convex if, given any $x, y \in \mathrm{dom}\,h$, then
 \begin{equation}\label{def:convex}
  h(\lambda x + (1-\lambda)y) \leq \lambda h(x) + (1 - \lambda) h(y),
  ~ \forall ~ \lambda \in [0, 1],
 \end{equation}


 \item[$(b)$] quasiconvex if, given any $x, y \in \mathrm{dom}\,h$, then
 \begin{equation}\label{def:qcx}
  h(\lambda x + (1-\lambda) y) \leq \max \{h(x), h(y)\}, ~ \forall ~
  \lambda \in [0, 1].
 \end{equation}
\end{itemize}
 \noindent It is said that $h$ is strictly convex (resp. strictly quasiconvex)
 if the inequa\-li\-ty in \eqref{def:convex} (resp. \eqref{def:qcx}) is strict
 whenever $x \neq y$ and $\lambda \in \, ]0, 1[$. Every convex function is
 quasiconvex, but the reverse statement does not hold as the function
 $h(x) = x^{3}$ shows.   Recall that
 \begin{align*}
  h ~ \mathrm{is ~ convex} & \Longleftrightarrow \, \mathrm{epi}\,h ~
  \mathrm{is ~ a ~ convex ~ set;}\\
  h ~ \mathrm{is ~ quasiconvex} & \Longleftrightarrow \, S_{\lambda} (h) ~
  \mathrm{is ~ a ~ convex ~ set ~ for ~ all ~ } \lambda \in \mathbb{R}.
 \end{align*}

Quasiconvex functions appear in many applications from different fields
as, for instance, in Economics and Financial Theory, especially in consumer
preference theory (see \cite{D-1959,MWG}), since quasiconcavity is the
mathematical formulation of the natural assumption of a {\it tendency to
diversification} on the consumers.

It is said that $h$ is quasilinear if $h$ is quasiconvex and $-h$ is quasiconvex.
As a consequence,  its sublevel set $S_{\lambda} (h)$ and its upper level
sets $W_{\lambda} (h)$  are convex for all $\lambda \in \mathbb{R}$ (see
\cite[Theorem 3.3.1]{CMA}).
Note that every nonincreasing/nondecreasing convex (or concave)
function is quasilinear, because all its sublevel and upper level sets are convex.

Let $K \subseteq \mathbb{R}^{n}$ be a convex set and $h: K \rightarrow
\mathbb{R}$   a differentiable function. Then the following assertions hold:
\begin{itemize}
 \item[$(i)$] $h$ is quasiconvex if and only if for every $x, y \in K$, we have
 (see \cite{AE} and \cite[Theorem 3.11]{ADSZ}) that
\begin{equation}\label{char:AE}
 h(x) \leq h(y) ~ \Longrightarrow ~ \langle \nabla h(y), x - y \rangle
 \leq 0;
\end{equation}

 \item[$(ii)$] $h$ is quasilinear if and only if for every $x, y \in K$, we have
 (see \cite[Theorem 3.3.6]{CMA})
 \begin{equation}\label{char:quasilinear}
  h(x) = h(y) ~ \Longrightarrow ~ \langle \nabla h(y), x - y \rangle
  = 0.
 \end{equation}
 \end{itemize}

 Let $h: \mathbb{R}^{n} \rightarrow \mathbb{R}$ be a differentiable
 function. Then $h$ is said to be pseudoconvex (see \cite{Manga}) if
\begin{equation}\label{pseudoconvex}
 h(x) < h(y) ~ \Longrightarrow ~ \langle \nabla h(y), x - y \rangle < 0.
\end{equation}
A function $h$ is pseudoconcave if $-h$ is pseudoconvex.

If $h$ is pseudoconvex, then every local minimum is global minimum
\cite[Theorem 3.2.5]{CMA} and, as a consequence, if $x$ is not a local
minimum point of a pseudoconvex function $h$, then
\begin{equation}\label{neatly:qcx}
 {\rm int}(S_{h(x)} (h)) = S^{<}_{h(x)} (h).
\end{equation}

For further study on generalized convexity we refer to
\cite{ADSZ,AE,CMA,Manga,SZ} among others.

\section{The Optimization Problem}\label{sec:03}

Based on \cite{angeris2022constant}, the following optimization problem,
modeling how to choose a valid trade, is considered
\begin{equation*}
 \begin{array}{lll}
  (Q): & \text{maximize} & U(x-y) \\
  & \text{subject to} & \varphi (R+\Gamma y-x)=\varphi (R)\text{, }R +
  \Gamma y-x \geq 0\text{, }R\geq x\geq 0\text{, }y\geq 0.%
 \end{array}%
\end{equation*}
Here $\varphi: \mathbb{R}^{n}_{+} \rightarrow \mathbb{R}_{+}$ is the
trading function  while $U: \mathbb{R}^{n}\rightarrow \mathbb{R}$ is the
utility that the trader wants to maximize, $x, y \in \mathbb{R}^{n}_{+}$
correspond with the received and given (tender) basket. In fact, this model is a generalization of the
model presented in \cite{angeris2022constant}, since it considers a diagonal
matrix $\Gamma =\left\{ \gamma _{i}\delta _{ij}\right\} _{1\leq i,j\leq n}$,
here $\delta_{ij}$ is the usual Kronecker delta ($\delta _{ii}=1,$ $\delta_{ij} =
0$ for $i\neq j$), the scalars $0 < \gamma _{i} < 1$ represent
the positive discount rate to the asset $i$, while $R \in \mathbb{R}_{+}^{n}$
is the reserve of available assets. Without loss of generality, we assume that
$R>0$ and
\begin{align}
 & \nabla U(x) \geq 0, ~ \forall ~ x \in \mathbb{R}^{n}, \label{200224a} \\
 & \nabla \varphi(x) \gneq 0, ~ \forall ~ x > 0. \label{200224b}
\end{align}
Note that \eqref{200224b} implies that $\varphi $ is strictly increasing with
respect to at least one component and, in particular,
\begin{equation}\label{strict:increa}
 x \geq y ~ \Longrightarrow ~ \varphi (x) \geq \varphi (y), ~ \forall ~
 x, y > 0.
\end{equation}

We can rewrite the problem in the following abstract way
\begin{equation*}
 \begin{array}{lcc}
  (Q): & \text{maximize} & f(x,y) \\
  & \text{s.t.} & h(x,y)=0\text{, }g(x,y)\leq 0,\text{ }\left( x,y\right) \in
  \hat{S},%
 \end{array}%
\end{equation*}
where the maps $f, h: \mathbb{R}^{n} \times \mathbb{R}^{n} \rightarrow
\mathbb{R}$, $g: \mathbb{R}^{n} \times \mathbb{R}^{n} \rightarrow \mathbb{R}^{n}$ and the convex set $\hat{S}$ are defined by
\begin{equation*}
\begin{array}{l}
f(x,y)=U(x-y), \\
h(x,y)=\varphi (R+\Gamma y-x)-\varphi (R), \\
g(x,y)=-R-\Gamma y+x, \\
\hat{S}=\{(x,y)\in \mathbb{R}^{n}\times \mathbb{R}^{n}:\text{ }R\geq x\geq 0%
\text{, }y\geq 0\}.
\end{array}
\end{equation*}

This is an optimization problem with an equality and inequality
constraint and convex constraints on the variable. While natural assumptions
can be given to  $U$ and $\varphi$, we will assume that the
maps are continuously differentiable.

We first study the solvability of this problem in the general case when
the objective map is continuous and the feasible set is nonvoid and compact, which we assume as a hypothesis,
therefore there exists a solution $(\bar{x}, \bar{y})$ to $(Q)$ by applying
known results as, for instance, \cite[Theorem 2.3]{jahn2020introduction}.
Furthermore we assume the following complementary condition on
$(\bar{x}, \bar{y})$
\begin{equation}\label{def:comple}
 0 \leq \bar{x} \bot \bar{y}\geq 0 ~ \Longleftrightarrow ~ \bar{x} \geq 0, \,
 \bar{y} \geq 0\text{, }\bar{x}_{i} \bar{y}_{i} = 0, ~ \forall ~ i=1,...,n.
\end{equation}
In the context of our application, this is a natural assumption since it
corres\-ponds with the non-overlapping support of valid tender and receive
baskets (see \cite[Section 3]{angeris2022constant}). In the following we
understand that $U$ is strongly increasing when $x \gneq y$ implies
$U(x) > U(y)$.

\begin{pr}\label{230224}
 There exists a solution $(\bar{x}, \bar{y}) $ of $(Q)$. Furthermore,
 if $U$ is strongly increasing, then $0\leq \bar{x}\bot \bar{y}\geq 0$.
\end{pr}

Note that in virtue of \cite[Section 3]{angeris2022constant},
the complementary condition also justifies the constraint $x\leq R$ in $(Q)$.

\subsection{Necessary Conditions}\label{subsec:3-1}

A necessary condition is given by a standard multiplier rule
\cite[Theorem 5.3]{jahn2020introduction}: Assume $(\bar{x}, \bar{y})$ solves
problem $(Q)$. If the following property (Kurcyusz-Robinson-Zowe constraint
qualification) is verified
\begin{equation}
\left(
\begin{array}{c}
\nabla g(\bar{x},\bar{y}) \\
\nabla h(\bar{x},\bar{y})%
\end{array}%
\right) \cone(\hat{S}-(\bar{x},\bar{y}))+\cone\left(
\begin{array}{c}
\mathbb{R}_{+}^{n}+\left\{ g(\bar{x},\bar{y})\right\} \\
0%
\end{array}%
\right) =\mathbb{R}^{n}\times \mathbb{R},  \label{271023}
\end{equation}%
then the following multiplier rule is verified: Find $(\bar{x}, \bar{y}, \mu,
\lambda) \in \mathbb{R}^{n} \times \mathbb{R}^{n} \times
\mathbb{R}_{+}^{n}\times \mathbb{R}$ such that
\begin{equation}
 \begin{array}{l}
 \left( -\nabla f(\bar{x},\bar{y})+\mu ^{T} \nabla g(\bar{x}, \bar{y}) +
 \lambda \nabla h(\bar{x}, \bar{y}) \right) (x - \bar{x}, y - \bar{y}) \geq 0
 \text{ for every } x, y \in \hat{S}, \\
 \mu^{T} g(\bar{x}, \bar{y}) = 0 \\
 h(\bar{x}, \bar{y})=0. \label{2202243}
 \end{array}
\end{equation}

In the following result, we present a general necessary condition for
complementary solutions of problem $(Q)$. In this sense, given $x \in
\mathbb{R}^{n}$, we will use the notation $x_{-i} \equiv
(x_{1},...,x_{i-1,}x_{i+1},...,x_{n})$ such that $(x_{i},x_{-i}) = x$.

\begin{thr} \label{250224}
 Let $U$ and $\varphi$ be continuously differentiable functions such that
 conditions \eqref{200224a}  and \eqref{200224b} hold, and assume that
 there exists a complementary solution $(\bar{x},\bar{y})\in \hat{S}$,
 $0 \leq \bar{x} \bot \bar{y}\geq 0$, solving $(Q).$ Then, there exists a
 non-negative scalar $\alpha \geq 0$ such that the following conditions are verified:
\begin{small}
\begin{subequations}
\label{2202245}
\begin{align}
\dfrac{\partial U}{\partial x_{i}}(\bar{x}_{i},\bar{x}_{-i}-\bar{y}_{-i})&
\geq \alpha \dfrac{\partial \varphi }{\partial x_{i}}(R_{i}-\bar{x}%
_{i},R_{-i}+\gamma _{-i}\bar{y}_{-i}-\bar{x}_{-i})\text{ when }\bar{x}%
_{i}=R_{i}\text{, }\bar{y}_{i}=0\text{,}\medskip  \label{2202245i} \\
\dfrac{\partial U}{\partial x_{i}}(\bar{x}_{i},\bar{x}_{-i}-\bar{y}_{-i})&
=\alpha \dfrac{\partial \varphi }{\partial x_{i}}(R_{i}-\bar{x}%
_{i},R_{-i}+\gamma _{-i}\bar{y}_{-i}-\bar{x}_{-i})\text{ when }R_{i}>
\bar{x}_{i}>0\text{, }\bar{y}_{i}=0\text{,}\medskip  \label{2202245ii} \\
\dfrac{\partial U}{\partial x_{i}}(-\bar{y}_{i},\bar{x}_{-i}-\bar{y}_{-i})&
=\alpha \gamma _{i}\dfrac{\partial \varphi }{\partial x_{i}}(R_{i}+\gamma
_{i}\bar{y}_{i},R_{-i}+\gamma _{-i}\bar{y}_{-i}-\bar{x}_{-i})\text{ when }%
\bar{x}_{i}=0\text{, }\bar{y}_{i}>0.\text{ }\medskip  \label{2202245iii} \\
\alpha \dfrac{\partial \varphi }{\partial x_{i}}(R+\Gamma\bar{y}-\bar{x})& \geq \dfrac{\partial U}{%
\partial x_{i}}(\bar{x}-\bar{y})\geq \alpha \gamma _{i}\dfrac{\partial \varphi }{\partial
x_{i}}(R+\Gamma\bar{y}-\bar{x})\text{ when }\bar{x}_{i}=\bar{y}_{i}=0,\medskip  \label{2202245iv}
\\
\varphi (R+\Gamma \bar{y}-\bar{x})& =\varphi (R).  \label{2202245v}
\end{align}
\end{subequations}
\end{small}
\end{thr}

\begin{proof}
 In the first place, by a direct computation
\begin{equation*}
\begin{array}{l}
\nabla f(\bar{x},\bar{y})=\left(
\begin{array}{cc}
\nabla U(\bar{x}-\bar{y}) & -\nabla U(\bar{x}-\bar{y})%
\end{array}%
\right) \\
\nabla g(\bar{x},\bar{y})=\left(
\begin{array}{cc}
I_{n} & -\Gamma%
\end{array}%
\right) \\
\nabla h(\bar{x},\bar{y})=\left(
\begin{array}{cc}
-\bar{P}^{T} & \bar{P}^{T}\Gamma%
\end{array}%
\right)%
\end{array}%
\end{equation*}%
where $I_{n}\in \mathbb{R}^{n\times n}$ denotes the identity matrix, $\bar{R}
= R+\Gamma \bar{y}-\bar{x},$ and
\begin{equation*}
\bar{P}^{T} := \nabla \varphi (\bar{R}) = \nabla \varphi (R + \Gamma \bar{y}
- \bar{x}),
\end{equation*}
is considered as a column vector following a usual convention for vectors
while $\bar{P} \gneq 0$ by hypothesis \eqref{200224b}. This notation
reflects that $\bar{P}$ can be interpreted as a vector of prices.

We prove now that the constraint qualification \eqref{271023} is verified.
Condition \eqref{271023} is equivalent to the verification of the two
following conditions: For every $(z,t)\in \mathbb{R}^{n}\times \mathbb{R}$
we can take $\theta ,\beta \in \mathbb{R}_{+}$, $x,y\in \hat{S}$, $c,d\in
\mathbb{R}_{+}^{n},$
\begin{subequations}
\begin{align}
& z=\theta  \left[ (x-\bar{x})-\Gamma (y-\bar{y})\right] +\beta (c-\bar{R}),
\label{180224a} \\
& t=\theta  (\bar{P}^{T}(x-\bar{x})-\bar{P}^{T}\Gamma (y-\bar{y})).
\label{180224b}
\end{align}%
\eqref{180224b} holds if $\bar{P}\neq 0$, which follows from \eqref{200224b}%
. For the first equality, \eqref{180224a}, by the complementary condition,
either $\bar{R}_{i}=R_{i}-\bar{x}_{i}$ or $\bar{R}_{i}=R_{i}+\gamma _{i}\bar{%
y}_{i}>R_{i}>0$. In this sense, we define
\end{subequations}
\begin{equation*}
\begin{array}{l}
I^{a}:=\{i\in \{1,...,n\}:\,\bar{x}_{i}=R_{i}\}, \\
I^{b}:=\{i\in \{1,...,n\}:\,0<\bar{x}_{i}<R_{i}\}, \\
I^{c}:=\{i\in \{1,...,n\}:\,0<\bar{y}_{i}\}\text{, }\  \\
I^{d}:=\{i\in \{1,...,n\}:\,0=\bar{x}_{i}=\bar{y}_{i}\},%
\end{array}%
\end{equation*}%
which is a partition on the index set $\{1,...,n\}$ and such that
\begin{equation*}
\bar{R}_{i}=\left\{
\begin{array}{lll}
0, & \text{if} & i\in I^{a}, \\
R_{i}-\bar{x}_{i}, & \text{if} & i\in I^{b}, \\
R_{i}+\gamma _{i}\bar{y}_{i}, & \text{if} & i\in I^{c}, \\
R_{i} & \text{if} & i\in I^{d}.%
\end{array}%
\right.
\end{equation*}

From this, the following useful property is straightforward
\begin{equation}  \label{190224}
\bar{R}_{i}=0 ~ \Longleftrightarrow ~ i \in I^{a}.
\end{equation}%
To prove \eqref{180224a} we consider two possibilities:

\begin{itemize}
\item If $I^{a}=\emptyset $,  then $0\leq \bar{x}_{i}<R_{i}$
for every $i\in \{1,...,n\}$, thus $\bar{R}_{i}>0$ and $\bar{R}\in \interior%
(\mathbb{R}_{+}^{n})$. Then for every $z\in \mathbb{R}^{n}$ we can take $%
\beta >0$ small enough such that $c:=\bar{R}+\frac{1}{\beta }z\in \mathbb{R}%
_{+}^{n}$. Hence, by taking $x=\bar{x}$, $y=\bar{y}$, we have
\begin{equation*}
  \left[ (\bar{x}-\bar{x})-\Gamma (\bar{y}-\bar{y})\right] +\beta (\bar{%
R}+\frac{1}{\beta }z-\bar{R})=z,
\end{equation*}%
and \eqref{180224a} holds. Note that this analysis is valid for every $%
\theta \geq 0$, which is fundamental for the second case.

\item If $I^{a}\neq \emptyset $, then we proceed as follows:
we apply the previous reasoning line when $i\notin I^{a}$, while in the
other case, we need to verify:
\begin{align}
z_{i}& =\theta ((x_{i}-\bar{x}_{i})-\gamma _{i}(y_{i}-\bar{y}_{i}))+\beta
(c_{i}-\bar{R}_{i}),~\forall ~i\in I^{a}  \notag \\
& =\theta ((x_{i}-R_{i})-\gamma _{i}y_{i})+\beta c_{i}~(\text{because }i\in
I^{a}).  \label{q1}
\end{align}%
If $z_{i}<0$, then we take $c_{i}=0$, $x_{i}=y_{i}=R_{i}$ and $\theta  =-%
\frac{z_{i}}{\gamma _{i}R_{i}}>0$; while if $z_{i}\geq 0$, then we take $%
x_{i}=R_{i}$, $y_{i}=0$, $\beta =1$ and $c_{i}=z_{i}$.
\end{itemize}

On the other hand, the necessary optimality conditions \eqref{2202243} are
given by
\begin{equation}
\begin{array}{l}
- \nabla U(\bar{x} - \bar{y})^{T} (x - \bar{x}) + \nabla U(\bar{x} - \bar{y}%
)^{T} (y - \bar{y}) + \mu ^{T} (x - \bar{x}) - \mu ^{T} \Gamma (y-\bar{y})
\\
- \lambda \bar{P}^{T} (x - \bar{x}) + \lambda \bar{P}^{T} \Gamma (y - \bar{y}%
) \geq 0, ~ \forall ~ x, y \in \hat{S}, \\
\mu^{T} (-\bar{R})=0, \\
\varphi (\bar{R}) = \varphi (R).%
\end{array}
\label{cond:holds}
\end{equation}

Applying \eqref{190224}, since $\mu ^{T}(-\bar{R})=0$ we have $%
\mu _{i}=0$ for every $i \not\in I^{a}$, thus
\begin{align}
\mu^{T} (x - \bar{x}) - \mu^{T} \Gamma (y - \bar{y}) & = \sum\limits_{i \in
I^{a}} \mu_{i} \left( x_{i} - \bar{x}_{i} \right) - \mu _{i} \gamma _{i}
\left( y_{i} - \bar{y}_{i} \right)  \notag \\
& = \sum\limits_{i \in I^{a}} \mu_{i} \left( x_{i} - R_{i} \right) - \mu
_{i} \gamma_{i} y_{i} \leq 0,  \label{q2}
\end{align}
for every $x_{i} \leq R_{i}$, $y_{i} \geq 0$. Therefore
\begin{align*}
& - \nabla U(\bar{x} - \bar{y})^{T} (x - \bar{x}) + \nabla U(\bar{x} - \bar{y%
})^{T} (y - \bar{y}) - \lambda \bar{P}^{T} (x - \bar{x}) + \lambda \bar{P}%
^{T} \Gamma (y - \bar{y}) \\
& \hspace{1.0cm} \geq -\mu^{T} (x - \bar{x}) + \mu^{T} \Gamma (y - \bar{y})
\geq 0, ~ \forall ~ x_{i} \leq R_{i}, ~ \forall ~ y_{i} \geq 0.
\end{align*}
Hence,
\begin{equation}
(- \nabla U(\bar{x} - \bar{y}) - \lambda \bar{P})^{T} (x - \bar{x}) +
(\nabla U(\bar{x} - \bar{y}) + \lambda \Gamma \bar{P})^{T} (y - \bar{y})
\geq 0, ~ \forall ~ x, y \in \hat{S}.  \label{sep:var}
\end{equation}

By taking $x_{i} = \bar{x}_{i}$ and $y_{i}=\bar{y}_{i}$ on each index
subset, condition \eqref{sep:var} can be equivalently decomposed into four
conditions,
\begin{subequations}
\begin{align}
\sum\limits_{i \in I^{a}} \left( - \nabla U(\bar{x} - \bar{y})_{i} - \lambda
\bar{P}_{i} \right)   (x_{i} - R_{i}) + \left( \nabla U(\bar{x}
- \bar{y})_{i} + \lambda \gamma _{i} \bar{P}_{i} \right) y_{i} & \geq 0,
\label{1902245I} \\
\sum\limits_{i \in I^{b}} \left( - \nabla U (\bar{x} - \bar{y})_{i} -
\lambda \bar{P}_{i} \right) (x_{i} - \bar{x}_{i}) + \left( \nabla U (\bar{x}
- \bar{y})_{i} + \lambda \gamma_{i} \bar{P}_{i} \right) y_{i} & \geq 0,
\label{1902245II} \\
\sum\limits_{i\in I^{c}}\left( -\nabla U(\bar{x}-\bar{y})_{i}-\lambda \bar{P}%
_{i}\right) x_{i}+\left( \nabla U(\bar{x}-\bar{y})_{i}+\lambda \gamma _{i}%
\bar{P}_{i}\right) (y_{i}-\bar{y}_{i})& \geq 0\text{,}  \label{1902245III} \\
\sum\limits_{i\in I^{d}}\left( -\nabla U(\bar{x}-\bar{y})_{i}-\lambda \bar{P}%
_{i}\right) x_{i}+\left( \nabla U(\bar{x}-\bar{y})_{i}+\lambda \gamma _{i}%
\bar{P}_{i}\right) y_{i}& \geq 0\text{,}  \label{1902245IV}
\end{align}%
for every  $R_{i}\geq x_{i}\geq 0$, $y_{i}\geq 0$.

Let us now prove that these four conditions imply the conditions in the
optimality system \eqref{2202245}.

For the case \eqref{1902245II}, when $I^{b} \neq \emptyset$, the quantity
\begin{equation*}
\sum\limits_{i \in I^{b}} \left( - \nabla U (\bar{x} - \bar{y})_{i} -
\lambda  \bar{P}_{i} \right) (x_{i} - \bar{x}_{i}),
\end{equation*}
can be arbitrarily positive or negative as $\bar{x}_{i} \in \, ]0, R_{i}[$
and  $x_{i} \in [0, R_{i}]$ is arbitrary, hence necessarily $\sum\limits_{i
\in I^{b}}  \left( - \nabla U (\bar{x} - \bar{y})_{i} - \lambda \bar{P}_{i}
\right) (x_{i} -  \bar{x}_{i}) = 0$  for every $\bar{x}_{i} \in
\, ]0, R_{i}[$ and  every $x_{i} \in [0, R_{i}]$  and, since $\sum\limits_{i
\in I^{b}} \left(\nabla  U (\bar{x} - \bar{y})_{i} + \lambda \gamma _{i}
\bar{P}_{i} \right) y_{i} \geq 0$  for every $y_{i} \geq 0$, we obtain
\end{subequations}
\begin{equation*}
- \nabla U (\bar{x} - \bar{y})_{i} - \lambda \bar{P}_{i} = 0, ~ \nabla U (%
\bar{x} - \bar{y})_{i} + \lambda \gamma _{i} \bar{P}_{i} \geq 0, ~ \forall ~
i \in I^{b}.
\end{equation*}
From this, $-\nabla U(\bar{x}-\bar{y})_{i}=\lambda \bar{P}_{i}$, which
implies $\lambda \leq 0$ because $\bar{P}_{i}\geq 0$ and $\nabla U (\bar{x}
- \bar{y})_{i} \geq 0$ by assumptions \eqref{200224a} and \eqref{200224b}.
Furthermore, the second condition always holds, since $\nabla U (\bar{x} -
\bar{y})_{i} + \lambda \gamma_{i} \bar{P}_{i} = \left\vert \lambda
\right\vert \bar{P}_{i} - \left\vert \lambda \right\vert \gamma _{i} \bar{P}%
_{i} = \left\vert \lambda \right\vert \bar{P}_{i} (1 - \gamma _{i}) \geq 0$,
consequently
\begin{equation*}
\nabla U (\bar{x} - \bar{y})_{i} = \left\vert \lambda \right\vert \bar{P}%
_{i}, ~ \forall ~ i \in I^{b},
\end{equation*}
when $I^{b}\neq \emptyset $, and this implies condition \eqref{2202245ii}.

By following a similar reasoning, from \eqref{1902245III} we have $-\nabla U
(\bar{x} - \bar{y})_{i} - \lambda \bar{P}_{i} \geq 0$ and $\nabla U (\bar{x}
- \bar{y})_{i} = - \lambda \gamma _{i} \bar{P}_{i}$ for all $i \in I^{c}$,
so $\lambda \leq 0$ and $-\nabla U(\bar{x} - \bar{y})_{i} - \lambda \bar{P}%
_{i} = - \left\vert \lambda \right\vert \gamma _{i} \bar{P}_{i} + \left\vert
\lambda \right\vert \bar{P}_{i} \geq 0$, consequently \eqref{1902245III}
implies
\begin{align*}
& \nabla U (\bar{x} - \bar{y})_{i} = \left\vert \lambda \right\vert \gamma
_{i} \bar{P}_{i}, ~ \forall ~ i \in I^{c},
\end{align*}
when $I^{c}\neq \emptyset $, and \eqref{2202245iii} is verified.

In the same way, if $I^{d}\neq \emptyset $, then $\lambda \leq 0$, $\bar{x}_{i}
- \bar{y}_{i} = 0$ and \eqref{1902245IV} implies
\begin{align*}
& - \nabla U (\bar{x} - \bar{y})_{i} - \lambda \bar{P}_{i} \geq 0, ~ \nabla
U (\bar{x} - \bar{y})_{i} + \lambda \gamma _{i} \bar{P}_{i} \geq 0, \\
& \Longleftrightarrow ~ \left\vert \lambda \right\vert \bar{P}_{i} \geq
\nabla U (\bar{x} - \bar{y})_{i} \geq \left\vert \lambda \right\vert \gamma
_{i} \bar{P}_{i}, ~ \forall ~ i \in I^{d}.
\end{align*}

Furthermore, we necessarily have that one of the previous cases is verified,
i.e., $I^{b}\cup I^{c}\cup I^{d}\neq \emptyset $. Indeed, suppose for the
contrary that $\bar{x}_{i}=\bar{R}_{i}$ for every $i\in \{1,...,n\}$, then $%
\varphi (0)=\varphi (R)$, a contradiction to \eqref{200224b}. Hence, we can
always consider that $\lambda \leq 0$ and \eqref{1902245I} implies
\begin{align*}
& \hspace{0.5cm}-\nabla U(\bar{x}-\bar{y})_{i}-\lambda \bar{P}_{i}\leq
0,\,\nabla U(\bar{x}-\bar{y})_{i}+\lambda \gamma _{i}\bar{P}_{i}\geq 0 \\
& \Longrightarrow ~\nabla U(\bar{x}-\bar{y})_{i}\geq \left\vert \lambda
\right\vert \bar{P}_{i}=\max \{\left\vert \lambda \right\vert \gamma _{i}%
\bar{P}_{i},\left\vert \lambda \right\vert \bar{P}_{i}\},~\forall ~i\in
I^{a},
\end{align*}%
which corresponds with condition \eqref{2202245i}. This proves the desired
result.
\end{proof}

\begin{remark} \label{070324}
 From the proof, in Theorem \ref{250224} we can replace condition
 \eqref{200224a} by the following weaker assumption
 \begin{equation*}
  \nabla U(\bar{x}-\bar{y})\geq 0.
 \end{equation*}
\end{remark}

A direct consequence of Theorem \ref{250224} and Proposition \ref{230224} is
the following.

\begin{co}
 Let $U$ and $\varphi$ be continuously differentiable functions such that
 \eqref{200224b} is verified. If $U$ is strongly increasing, then every solution
 $(\bar{x}, \bar{y}) \in \hat{S}$ of (Q) verifies optimality system
 \eqref{2202245}.
\end{co}

\subsection{Sufficient Conditions}\label{subsec:3-2}

In the following result, we provide sufficient conditions without convexity
assumptions, neither on the utility function $U$ nor the trade function
$\varphi$. To that end, let $S \subseteq \mathbb{R}^{n}$ and
$C \subseteq \mathbb{R}^{m}$ be two nonempty sets and $f:
\mathbb{R}^{n} \rightarrow \mathbb{R}^{m}$  a differentiable mapping.
It is said that $f$ is $C$-quasiconvex at $\overline{x} \in S$ with respect to $S$
if for all $x \in S$ (see \cite[Definition 5.12]{jahn2020introduction}), the
following implications holds:
\begin{equation}\label{C:quasiconvex}
 f(x) - f(\overline{x}) \in C ~ \Longrightarrow ~ f^{\prime} (\overline{x})
 (x - \overline{x}) \in C.
\end{equation}

The desired result is given below.

\begin{thr} \label{010124}
 Let $U$ and $\varphi $ be continuously differentiable functions such that $U$
 is pseudoconcave (thus $-U$ is pseudoconvex) and $\varphi $ is quasilinear. If
 $(\bar{x},\bar{y})\in \hat{S}$ verifies conditions  \eqref{2202245}, then it
 solves problem (Q).
\end{thr}

\begin{proof}
 In {Theorem \ref{250224}}, we have shown that if $(\bar{x}, \bar{y}) \in
 \hat{S}$ satisfies conditions \eqref{2202245}, then general conditions
 \eqref{2202243} are verified. Hence, we have relations $(5.22)$ and $(5.23)$
 in \cite[Theorem 5.14]{jahn2020introduction}.

 We only need to prove that $(-f,g,h)$ is $\widehat{C}$-quasiconvex at
 $(\bar{x}, \bar{y})$ on the set:
\begin{equation}
 \widehat{C} := \left( \mathbb{R}_{-} \backslash \{0\} \times \left( -
 \mathbb{R}_{+}^{n} + \mathrm{cone}\{-\bar{R}\} - \mathrm{cone}
 \{- \bar{R}\}\right) \times \{0\}\right),  \label{big:C}
\end{equation}
where we recall $\bar{R} = R + \Gamma \bar{y} - \bar{x}$ that is, we should
prove that
\begin{equation}
(-f,g,h)(x,y)-(-f,g,h)(\bar{x},\bar{y})\in \widehat{C}~\Longrightarrow
~(-f,g,h)^{\prime }(\bar{x},\bar{y})\left( (x,y)-(\bar{x},\bar{y})\right)
\in \widehat{C}.  \label{C:quasiconvex}
\end{equation}

We can simplify this expression. In this sense, since $\bar{R}\in \mathbb{R}%
_{+}^{n}$, $\mathrm{cone}\{-\bar{R}\}-\mathrm{cone}\{-\bar{R}\}={\langle
\bar{R}\rangle }$ is the subspace generated by $\bar{R}$, so
\begin{equation*}
\left( -\mathbb{R}_{+}^{n}+\mathrm{cone}\{-\bar{R}\}-\mathrm{cone}%
\{-R\}\right) =-\mathbb{R}_{+}^{n}+\langle \bar{R}\rangle =\mathbb{R}^{n}.
\end{equation*}
The last equality is a consequence of the Hahn-Banach theorem.
${-\mathbb{R}_{+}^{n}+\langle \bar{R}\rangle }$ is a convex set with
nonempty interior, on the contrary, if ${-\mathbb{R}_{+}^{n} + \langle
\bar{R}\rangle \neq \mathbb{R}^{n}}$ there is some $\tilde{x} \in
{\mathbb{R}^{n}}$ such that $\tilde{x} \notin {-\mathbb{R}_{+}^{n} +
\langle \bar{R}\rangle}$. By Hahn-Banach theorem (see \cite[Theorem
C.2]{jahn2020introduction} for instance), there exists some
$\beta \in \mathbb{R}^{n} \backslash \{0\}$ such that
\begin{equation*}
\left\langle \beta ,\tilde{x}\right\rangle \geq -\left\langle \beta
,c\right\rangle +\left\langle \beta ,\alpha \bar{R}\right\rangle \text{ for
every }\alpha \in {\mathbb{R}}\text{, }c\in {\mathbb{R}_{+}^{n}}.
\end{equation*}

From this we can deduce that $\beta \in {\mathbb{R}_{+}^{n}} \backslash
\{0\}$ and ${\langle \beta ,\bar{R}\rangle }=0$, the latter is an absurd since $\bar{R} > 0$
and  $\beta \neq 0$, consequently, ${-\mathbb{R}_{+}^{n} + \langle
\bar{R} \rangle = \mathbb{R}^{n}}$. Then, $\hat{C}=\mathbb{R}_{-}
\backslash \{0\}\times {\mathbb{R}^{n}\times \{0\}}$ and relation
\eqref{C:quasiconvex} is given by components as follows:
\begin{itemize}
 \item[$(i)$] $-f(x,y) - (-f(\bar{x}, \bar{y})) \in \mathbb{R}_{-}
 \backslash \{0\} \Rightarrow -f^{\prime} (\bar{x}, \bar{y}) (x, y) \in
 \mathbb{R}_{-} \backslash \{0\}$,

 \item[$(ii)$] $g(x,y) - g(0,0) \in {\mathbb{R}^{n}} \Rightarrow
 g^{\prime} (\bar{x}, \bar{y}) (x, y) \in {\mathbb{R}^{n}}$.

 \item[$(iii)$] $h(x,y) = h(\bar{x}, \bar{y}) \Rightarrow h^{\prime}
 (\bar{x}, \bar{y}) (x-\bar{x}, y-\bar{x}) = 0$.
\end{itemize}
Condition $(ii)$ is clearly trivial, in the following we prove the other two
conditions.

\begin{itemize}
\item[$(i)$] For $-f$: Let $(x,y)\in \widehat{S}$ ($x\neq y$ always). Then
\begin{align*}
(-f)(x,y)-(-f)(\bar{x},\bar{y})\in \mathbb{R}_{-}\backslash \{0\}&
\Longleftrightarrow ~f(x,y)>f(\bar{x},\bar{y}) \\
& \Longleftrightarrow -U(x-y)<-U(\bar{x}-\bar{y}) \\
& \overset{\eqref{pseudoconvex}}{\Longrightarrow }\,\langle -\nabla U(\bar{x}%
-\bar{y}),\left( x-\bar{x}\right) -\left( y-\bar{y}\right) \rangle <0 \\
& \Longleftrightarrow \langle \left( -\nabla U(\bar{x}-\bar{y}),\nabla U(%
\bar{x}-\bar{y})\right) ,(x-\bar{x},y-\bar{y})\rangle <0 \\
& \Longleftrightarrow \langle -\nabla f(\bar{x},\bar{y}),(x,y)-(\bar{x},\bar{%
y})\rangle <0 \\
& \Longleftrightarrow -f^{\prime }(\bar{x},\bar{y})(x-\bar{x},y-\bar{y})\in
\ \mathbb{R}_{-}\backslash \{0\},
\end{align*}
where basically we have applied the pseudoconvexity of $-U$.

\item[$(iii)$] For $h$: Let $(x,y)\in \widehat{S}$. Then,
\begin{align*}
 h(x, y) - h(\bar{x}, \bar{y}) = 0 & \Longleftrightarrow \varphi (R +
 \Gamma y - x) = \varphi (R + \Gamma \bar{y} - \bar{x}) \\
 & \overset{\eqref{char:quasilinear}}{\Longrightarrow }\langle \nabla
 \varphi (\bar{R}),\ \Gamma (y-\bar{y}) - \left( x-\bar{x}\right)
 \rangle =0, \\
 & \Longleftrightarrow \langle (-\nabla \varphi (\bar{R})^{T},\nabla
 \varphi (\bar{R})^{T} \Gamma), \left( x - \bar{x}, y - \bar{y} \right)
 \rangle =0 \\
 & \Longleftrightarrow h^{\prime }(\bar{x},\bar{y})\left( x - \bar{x},
 y - \bar{y} \right) = 0,
\end{align*}%
and this proves $(iii)$ holds.
\end{itemize}

We have proven that $(-f, g, h)$ is $\widehat{C}$-quasiconvex at $(\bar{x},
\bar{y})$, and consequently $(\bar{x}, \bar{y})$ is optimal for problem (Q)
by \cite[Corollary 5.15]{jahn2020introduction}. \end{proof}

\begin{remark}
 From the proof, in the previous theorem we can replace the quasilinearity of
 $\varphi $ by the following weaker condition
\begin{equation}
 \langle \nabla \varphi (\bar{R}), ~ \Gamma (y - \bar{y}) - \left( x - \bar{x}
 \right) \rangle = 0, ~ \forall ~ (x, y) \in \widehat{S}, ~~ \varphi (R +
 \Gamma y - x) = \varphi (R). \label{cond:h2}
\end{equation}
\end{remark}

\subsection{The Characterization Result}\label{subsec:3-3}

As a direct consequence of Theorems \ref{250224} and \ref{010124}, we
establish sufficient and necessary conditions for problem $(Q)$.

\begin{thr} \label{0703242}
 Let $U$ and $\varphi $ be continuously differentiable functions such that $U$
 is pseudoconcave (thus $-U$ is pseudoconvex) and $\varphi $ is quasilinear;
 and conditions \eqref{200224a} and \eqref{200224b} are verified. Suppose
 that $0 \leq \bar{x}\bot \bar{y}\geq 0$. Then, $(\bar{x}, \bar{y}) \in \hat{S}$
 solves $(Q)$ if and only if it verifies optimality system \eqref{2202243}.
\end{thr}

In particular we have the following general result.

\begin{co}
 Let us assume the same assumptions as in Theorem \ref{0703242}, and in
 addition let $U$ be strongly increasing. Then, $(\bar{x}, \bar{y}) \in \hat{S}$
 solves $(Q)$ if and only if it verifies optimality system \eqref{2202243}.
\end{co}

\section{No-trade Characterization}\label{sec:04}

A direct application of optimality system \eqref{2202245} is to study under
which conditions the unique solution is given by $(\bar{x},\bar{y})=(0,0)$,
this is known as the no-trade condition.

\begin{defn}
 It is said that a no-trade condition is verified if $(\bar{x}, \bar{y}) = (0,0)$ is
 the unique solution of problem (Q).
\end{defn}

No-trade condition means that trading does not increase the trader's
utility, that is, the trader does not propose any trade, see \cite%
{angeris2022constant}. In this sense, $P=\nabla \varphi (R)$ can be
interpreted as the (unscaled) prices vector, and if we divide it by $P_{n}>0$
(price of numeraire) we get the reported prices $p_{i}=P_{i}/P_{n}.$ In this
sense, following \cite{angeris2022constant}, we can define a no-trade set
\begin{equation}
{K_{\Gamma }}:=\{z\in \mathbb{R}^{n}:~\alpha \Gamma z\leq \nabla U(0)\leq
\alpha z\ \ \text{for some }\alpha \geq 0\},  \label{notrade:cone}
\end{equation}%
such that prices belonging to that set assures the no-trade property. Under the hypotheses of Theorem \ref{0703242}, when $(0,0)$ is a solution to the problem, we can choose $\alpha >0$ such that ${K_{\Gamma }=}\{z\in
\mathbb{R}^{n}:~\frac{1}{\alpha }\nabla U(0)\leq z\leq \frac{1}{\alpha }%
\Gamma ^{-1}\nabla U(0)\}$ is a compact convex set in $\mathbb{R}^{n}$. This is a consequence of Theorem \ref{250224}, when $(\bar{x%
},\bar{y})=(0,0)$ solves (P), and the complementary condition is trivially
verified, by Theorem \ref{250224}, optimality system \eqref{2202245} is
reduced to \eqref{2202245iv} which can be equi\-va\-len\-tly expressed by
the condition $p\in {K_{\Gamma }}$ .

\begin{pr}
 Let $U$ and $\varphi $ be continuously differentiable functions such that
 conditions \eqref{200224a} and \eqref{200224b} hold. No-trade condition
 implies $p\in {K_{\Gamma }}$.
\end{pr}

In fact, we can characterize the no-trade condition for general classes of
utility and trading functions by applying Theorem \ref{0703242}.

\begin{thr} \label{260224}
 Let $U$ and $\varphi $ be continuously differentiable functions such that
 $U$ is pseudoconcave (thus $-U$ is pseudoconvex) and $\varphi $ is
 quasilinear, and conditions \eqref{200224a} and \eqref{200224b} hold.
 The following properties are equivalent:
 \begin{itemize}
  \item[$(a)$] No-trade condition is verified.

  \item[$(b)$] $p \in K_{\Gamma }$.
 \end{itemize}
\end{thr}

\begin{remark}
 In \cite{angeris2022constant}, the trading function $\varphi $ is assumed
 convex, differentiable and increasing, thus given any $\lambda \in
 \mathbb{R}$, the sets $S_{\lambda} (\varphi)$ and $U_{\lambda}
 (\varphi)$ are convex (because $\varphi $ is increasing and convex), then
 $\varphi $ is quasilinear by \cite[Theorem 3.3.1]{CMA} (see the comment
 right after relation \eqref{char:quasilinear}). Hence, \cite[Section
 5.1]{angeris2022constant} is a particular case of our approach for the case
 of a single discount rate $\gamma_{i}=\gamma$.
\end{remark}

\section{Examples on Mean Functions}\label{sec:05}

\subsection{Weighted Quasi-arithmetic Means}\label{sec:wqam}

The trading function $\varphi(x)$ is usually considered a type of mean, such as a geometric or arithmetic mean. These two means have a well-defined
convexity, as the former is concave and the latter is both convex and concave. Convexity is a hypothesis employed in \cite{angeris2022constant} to establish
their result about the no-trade condition. However, our main result relies on quasi-linearity rather than convexity, which is a less strict assumption.
In particular, our result captures quasi-arithmetic means, which are quasi-linear but not necessarily either convex or concave.
Let us note that quasi-arithmetic means have been studied since almost one century ago~\cite{finetti,knopp,kolmogorov,nagumo}.

Let $f:\mathbb{R} \longrightarrow \mathbb{R}$ be a continuous and strictly monotonic function; hereafter it will be called the mean generator.
Then we define the weighted quasi-arithmetic mean
\begin{equation}\label{wqam}
\varphi(x)=f^{(-1)}\left[ \sum_{i=1}^{n} \omega_i f(x_i) \right],
\end{equation}
where the weights $\omega_i >0$, $i=1,\ldots,n$, and $\sum\limits_{i=1}^{n}\omega _{i}=1$. Under these assumptions, the weighted quasi-arithmetic mean
fulfils
$$
\min\{x_1, \ldots, x_n\} \le \varphi(x) \le \max\{x_1, \ldots, x_n\},
$$
where the equalities only hold whenever $x_1=\ldots=x_n$, see Theorem 82 in~\cite{hlp}.
If we moreover limit ourselves to increasing generators, then the weighted quasi-arithmetic mean enjoys all the classical properties
attributable to means~\cite{aczel}. The weighted arithmetic and geometric means correspond, respectively,
to the cases $f(y)=y$ and $f(y)=\ln(y)$. For this type of mean, we have the following result:

\begin{corollary}\label{corql}
Let
\begin{eqnarray}\nonumber
\varphi: \mathbb{R}^n_+ &\longrightarrow& \mathbb{R} \\ \nonumber
x=(x_1,\ldots,x_n) &\longmapsto& f^{(-1)}\left[ \sum_{i=1}^{n} \omega_i f(x_i) \right],
\end{eqnarray}
be a quasi-arithmetic mean generated by a function $f$ which is both continuously differentiable and strictly monotonic.
If we furthermore assume that $f' > 0$, then $\varphi$ is both continuously differentiable and quasi-linear.
\end{corollary}

\begin{proof}
Under these hypotheses, $\varphi$ is a well-defined quasi-arithmetic mean~\cite{hlp}, and since $f' > 0$ then $\varphi$ is increasing~\cite{aczel},
and consequently quasi-linear. Moreover, since $f$ is continuously differentiable, $f^{(-1)}$ is too by the inverse function theorem
(because $f' > 0$), and the result follows by the chain rule of differential calculus.
\end{proof}

\begin{remark}
Obviously, the continuous differentiability of $f$ along with the condition $f' > 0$ imply the strict monotonicity of the mean generator.
\end{remark}

This corollary shows that quasi-arithmetic means $\varphi$ generated by suitable functions $f$ are admissible trading functions within the
theoretical framework we have constructed.

If we further assume that $f \in \mathcal{C}^4$ and $f>0$, $f'>0$, $f''>0$ simultaneously,
then $\varphi$ is convex if and only if $f'/f''$ is concave; alternatively, if and only if $f' f'''/(f'')^2$
is increasing (Theorem 106, \cite{hlp}). It follows from the proof of this theorem that, correspondingly,
$\varphi$ is concave if and only if $f'/f''$ is convex, or equivalently if and only if $f' f'''/(f'')^2$ is decreasing.
It also follows from this proof that ``monotonicity'' is a sufficient condition for this theorem to hold, in the sense that if we replace
the condition $f'>0$ for $f'<0$, the rest of the statement remains intact. Note that, complementarily, we can select a $f \in \mathcal{C}^4$ that fulfils
$f>0$, either $f'>0$ or $f'<0$, and $f''>0$, and such that it is neither convex nor concave; therefore it is not covered by theory
developed in \cite{angeris2022constant}. However, since such a $f$ gives rise to a well-defined and increasing mean $\varphi$ whenever $f'>0$,
this means that, under this assumption, it is quasi-linear and consequently fulfils the hypotheses of our main result (by Corollary~\ref{corql}).
In particular, the no-trade condition for such a $\varphi$ is given by the no-trade set $K_\Gamma$ (see the previous section).

Let us now consider a family of particular cases of $f$, and correspondingly of $\varphi$. Departing from Chapter III in~\cite{hlp},
and following the previous paragraphs in this section, we can check that
\begin{eqnarray}\nonumber
f: \mathbb{R}_+ &\longrightarrow& \mathbb{R} \\ \nonumber
y &\longmapsto& (y+1)^p \ln(y+1),
\end{eqnarray}
for any fixed $p>1$, defines a mean $\varphi$ that is continuously differentiable and increasing (and therefore quasi-linear),
but neither convex nor concave. Moreover, it admits the explicit representation
\begin{equation}\label{tflo}
\varphi(x)= \exp \left( W_0 \left\{ p \left[ \sum_{i=1}^{n} \omega_i (x_i+1)^p \ln(x_i+1) \right] \right\}\bigg/p \right) -1,
\end{equation}
where $W_0(\cdot)$ is the principal branch of the Lambert W function, a special function that has been
studied from classical to modern times~\cite{bruijn,corless,euler,fsc,ps,wright}. This trading function
does not fulfil the hypotheses employed in \cite{angeris2022constant}, but however does fulfil the hypotheses employed by us in the present work.
Consequently, the no-trade condition translates for it into the no-trade region $K_\Gamma$. Note that one can use many variants of this function without
altering this result, such as the alternative
\begin{eqnarray}\nonumber
&& \varphi(x) =  -e^{-1/p} + \\ \nonumber
&& \exp \left( W_0 \left\{ p \left[ \sum_{i=1}^{n} \omega_i \{(x_i+e^{-1/p})^p \ln(x_i+e^{-1/p}) + e p\} \right]
-e^{-1} \right\}\bigg/p \right),
\end{eqnarray}
among uncountably many (for each fixed $p$) closely related possibilities, as it is immediate to check.

\subsection{Numerical Experiments}

In this subsection we illustrate numerically a CFMM based on a
quasi-arithmetic mean.  It is based on \eqref{tflo}, for the special case
of equal weights $\omega _{i}=1/n$, $i=1,\ldots ,n$, and $p=2$, which we
denote by
\begin{equation*}
\varphi _{\mathrm{qm}}(x)=\exp \left( W_{0}\left\{ \frac{p}{n}\left[
\sum_{i=1}^{n}(x_{i}+1)^{2}\ln (x_{i}+1)\right] \right\} \bigg/p\ \right) -1.
\end{equation*}%
We use a similar example as that in \cite{angeris2022constant}, where $n=6$
assets are considered with reserves
\begin{equation*}
R=(1,3,2,5,7,6),
\end{equation*}%
and we also assume a single discount rate $\gamma _{i}=\gamma =0.9$ for
simplicity. The corresponding CFMMs prices are given by
\begin{equation*}
p_{i}=\frac{\nabla \varphi _{\mathrm{qm}}(R)_{i}}{\nabla \varphi _{\mathrm{qm%
}}(R)_{6}}\text{ for every }i\in \{1,...,6\}.
\end{equation*}%
We have computed the prices numerically by applying a simple forward finite
difference formula to find
\begin{equation*}
p=(0.13937573, 0.44068816, 0.28010876, 0.80312261, 1.20524975, 1).
\end{equation*}%
We consider a linear utility $U(z)=\pi ^{T}z$ in $(Q)$, where $\pi \in
\mathbb{R}_{+}^{6}$ models the trader private prices. In the first
experiment the corresponding prices are parameterized by
\begin{equation*}
\pi \equiv (tp_{1},p_{2},...,p_{6}),
\end{equation*}%
for some $t\in \lbrack \frac{1}{2},2]$, where we omit superscripts for
simplicity. We compute the optimal trades by solving the corresponding
problem $(Q)$ for each parameter, and we identify when the no-trade
condition, $\bar{x}-\bar{y}=0$, is verified. In our case, we have solved
numerically the problem by using the SciPy optimization library in Python \cite{V}. Since
the value $t=1$ corresponds to the case for which the trader and market prices coincide, it
is expected that the no-trade region is located around this value. In Figure \ref%
{fig:wquasi_a_1_param_notrade} we represent graphically the optimal trade
and no-trade regions. We also reproduce the experiment for two parameters:
\begin{equation*}
\pi \equiv (tp_{1},sp_{2},...,p_{6}),
\end{equation*}%
for some $t,s\in \lbrack \frac{1}{2},2]$. We compare the results with those
obtained for the same experiments, but employing more standard market
functions, namely the arithmetic mean
\begin{equation*}
\varphi _{\mathrm{am}}(x)=\frac{1}{n}\sum_{i=1}^{n}x_{i},
\end{equation*}%
and the geometric mean
\begin{equation*}
\varphi _{\mathrm{gm}}(x)=\prod\limits_{i=1}^{n}x_{i}^{1/n},
\end{equation*}%
which was considered in \cite{angeris2022constant}. We can check, see
Figures \ref%
{fig:arithmean_1_param_notrade},
\ref{fig:geomean_1_param_notrade}, \ref{fig:wquasi_a_1_param_notrade}, \ref{fig:wquasi_a_2_param_notrade}, \ref{fig:geomean_2_param_notrade}, and \ref%
{fig:arithmean_2_param_notrade}, that the corresponding results for both maps
are very similar to the ones obtained for $\varphi _{\mathrm{qm}}$,
as was expected from our theoretical developments in the previous section.

\begin{figure}[h]
\begin{subfigure}[b]{0.5\linewidth}\centering
\includegraphics[width=\linewidth]{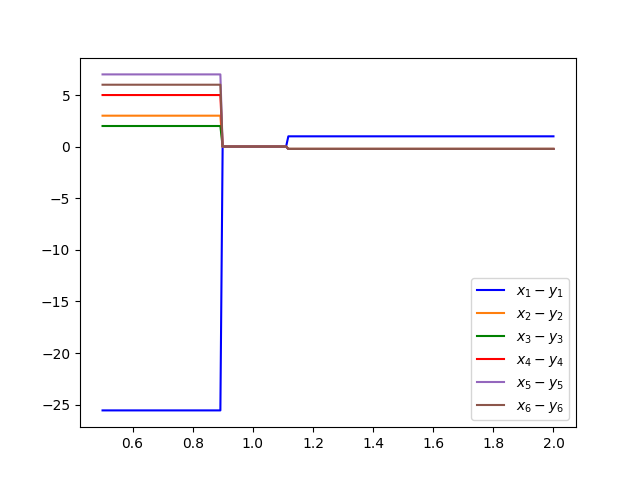}
 \caption{Optimal trades $x-y$ for $t\in[0.5,2]$. }
 \end{subfigure}
 \begin{subfigure}[b]{0.5\linewidth}\centering
\includegraphics[width=\linewidth]{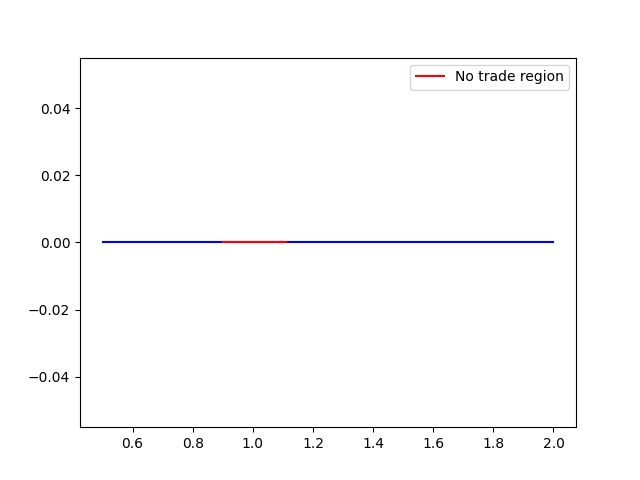}
\caption{No-trade region. }
\end{subfigure}
\caption{Market function $\varphi_{\mathrm{am}}$. One parameter case.}
\label{fig:arithmean_1_param_notrade}
\end{figure}

\begin{figure}[h]
\begin{subfigure}[b]{0.5\linewidth}\centering
\includegraphics[width=\linewidth]{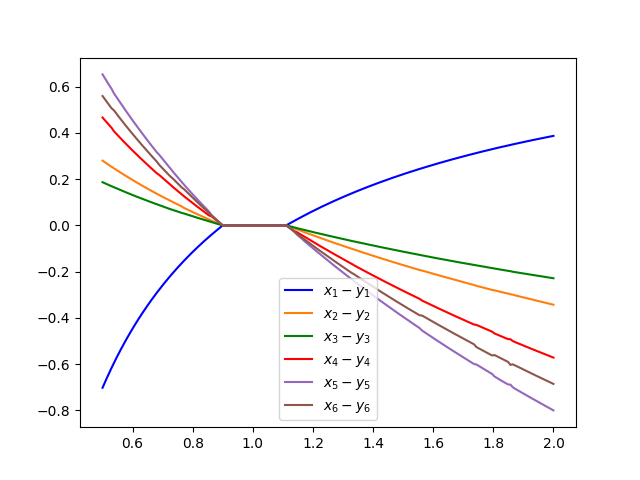}
 \caption{Optimal trades $x-y$ for $t\in[0.5,2]$. }
 \end{subfigure}
 \begin{subfigure}[b]{0.5\linewidth}\centering
\includegraphics[width=\linewidth]{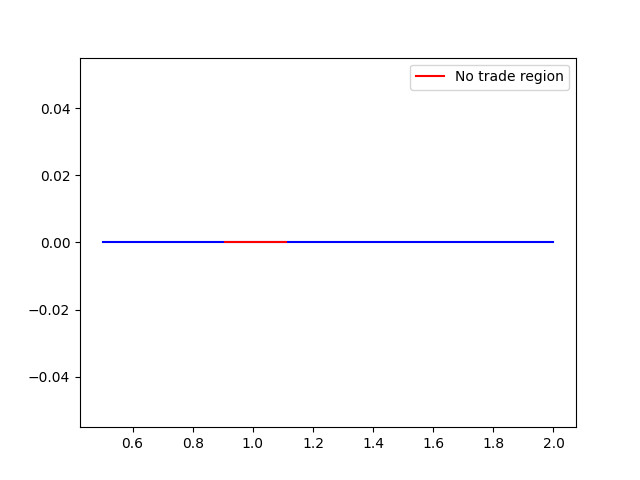}
\caption{No-trade region. }
\end{subfigure}
\caption{Market function $\varphi_{\mathrm{gm}}$. One parameter case.}
\label{fig:geomean_1_param_notrade}
\end{figure}

\begin{figure}[h]
\begin{subfigure}[b]{0.5\linewidth}\centering
\includegraphics[width=\linewidth]{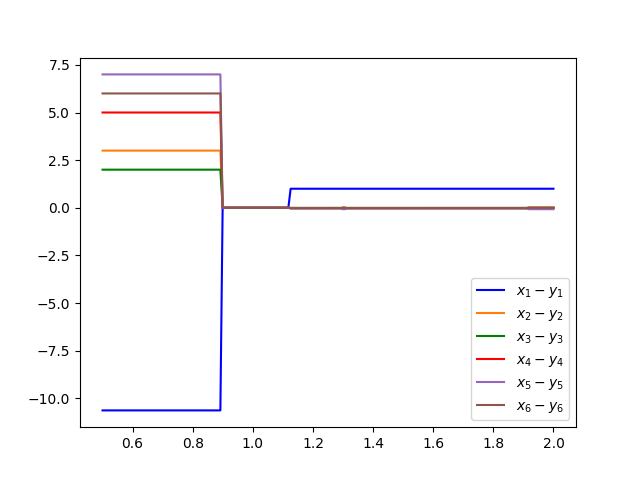}
 \caption{Optimal trades $x-y$ for $t\in[0.5,2]$. }
 \end{subfigure}
 \begin{subfigure}[b]{0.5\linewidth}\centering
\includegraphics[width=\linewidth]{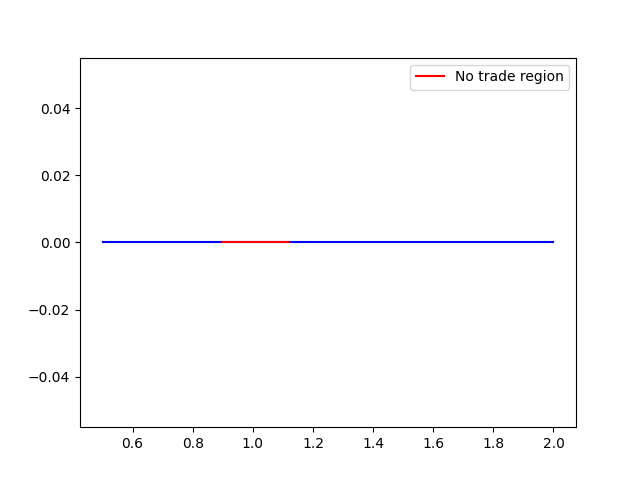}
\caption{No-trade region. }
\end{subfigure}
\caption{Market function $\varphi_{\mathrm{qm}}$. One parameter case.}
\label{fig:wquasi_a_1_param_notrade}
\end{figure}

\begin{figure}[h]
\includegraphics[width=\linewidth]{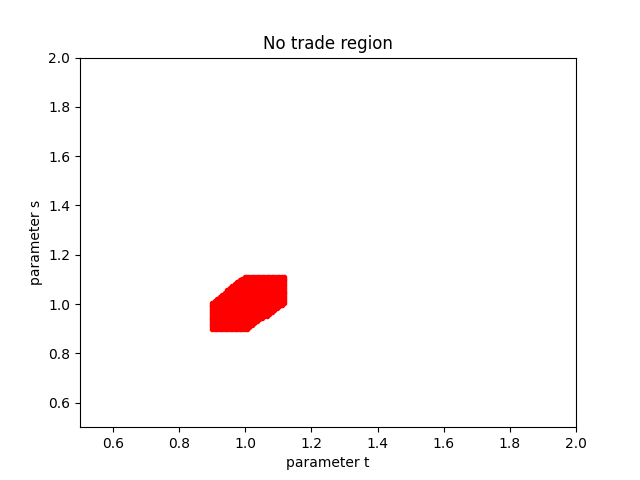}
\caption{Market function $\varphi_{\mathrm{qm}}$. Two parameters case.}
\label{fig:wquasi_a_2_param_notrade}
\end{figure}

\begin{figure}[h]
\includegraphics[width=\linewidth]{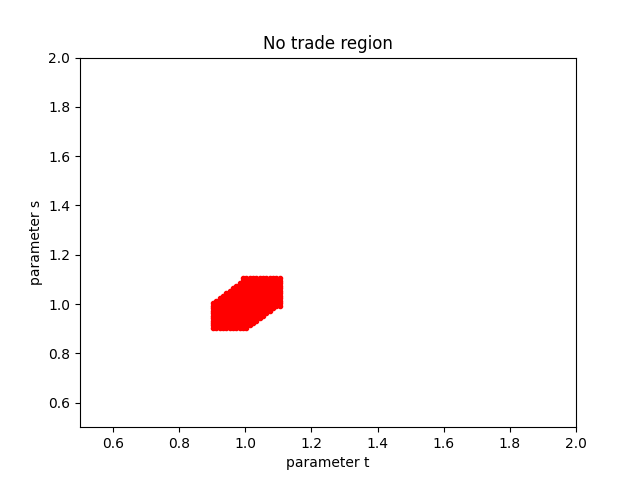}
\caption{Market function $\varphi_{\mathrm{gm}}$. Two parameters case.}
\label{fig:geomean_2_param_notrade}
\end{figure}

\begin{figure}[h]
\includegraphics[width=\linewidth]{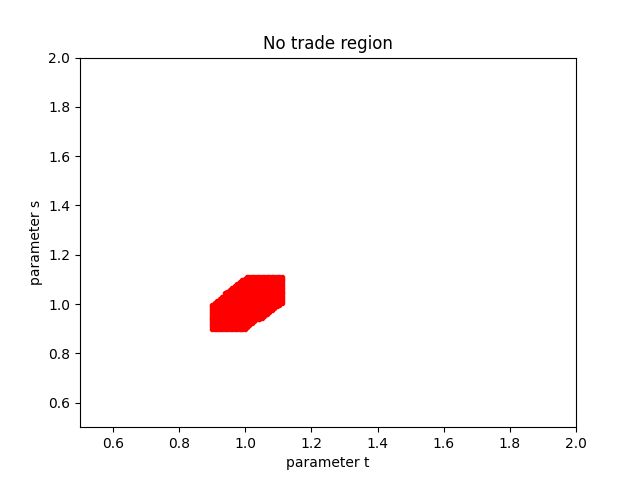}
\caption{Market function $\varphi_{\mathrm{am}}$. Two parameters case.}
\label{fig:arithmean_2_param_notrade}
\end{figure}

\subsection{Extensions}

Many other examples are still possible. For instance, consider the (unweighted) quasi-arithmetic mean
\begin{eqnarray}\nonumber
\varphi: \mathbb{R}^n_+ &\longrightarrow& \mathbb{R} \\ \nonumber
x=(x_1,\ldots,x_n) &\longmapsto& f^{(-1)}\left[ \frac{1}{n} \sum_{i=1}^{n} f(x_i) \right],
\end{eqnarray}
which is of course a particular case of~\eqref{wqam} (obtained by setting $\omega_i=1/n$ for $i=1,\ldots,n$).
The function $f$ is still assumed to be continuous and strictly monotonic, so the mean is well-defined. If we further assume that
$f \in C^2$ and $f' \neq 0$, then the convexity of the mean is characterized in Theorem 3.1 of~\cite{papa21} and concavity is characterized
in Theorem 2.2 of~\cite{papa18}. From now on we also assume that $f'>0$, so the mean is strictly increasing and covered by the statement of
Corollary~\ref{corql}. Since these characterizations require less regularity than those exposed in subsection~\ref{sec:wqam},
but the generators still fulfil the hypotheses of Corollary~\ref{corql} (what means the means are quasi-linear),
they open the possibility of constructing new trading functions included in our theory but not considered in \cite{angeris2022constant}.
Actually, even the weighted case~\eqref{wqam}, which was considered for $\mathcal{C}^4$ functions in~\cite{hlp}
(see subsection~\ref{sec:wqam}), can be extended to $C^2$ functions (Theorem 5, \cite{cdjj}).
Moreover, Theorem 1 of~\cite{cdjj} asserts the essential equivalence between the weighted and the unweighted cases in relation to convexity/concavity.
Finally, let us mention that this reference, \cite{cdjj}, presents some explicit examples of generators that give rise to means that are neither
convex nor concave and enjoy different degrees of smoothness. Consequently, they can be used as starting point to construct trading functions
regarded by our theory but not that of \cite{angeris2022constant}. One such mean, based on the smooth generator of Example 4 in~\cite{cdjj}, is given by
the explicit formula:
\begin{equation}\label{tflo2}
\varphi(x)= \sum_{i=1}^{n} \left[ \omega_i (x_i+e^{x_i}) \right] -W_0 \left\{ \exp \left[ \sum_{i=1}^{n} \omega_i (x_i+e^{x_i}) \right] \right\}.
\end{equation}

Extending the collection of examples is not just a matter of regularity. For
instance, in \cite{zhao}, the family of generalized quasi-arithmetic means
\begin{eqnarray}\nonumber
\Phi: \mathbb{R}_+^n &\longrightarrow& \mathbb{R} \\ \nonumber
x=(x_1,\ldots, x_n) &\longmapsto& F^{(-1)}\left[ \sum_{i=1}^{n} \omega_i f_i(x_i) \right],
\end{eqnarray}
is considered. In this case, the functions $f_i: \mathbb{R}_+ \longrightarrow \mathbb{R}$ are $C^2$ and convex
for all $i=1,\ldots, n$ (but no monotonicity assumption is in principle imposed). On the other hand, $F: \mathbb{R}_+ \longrightarrow \mathbb{R}$
is required to be convex, strictly increasing, and $C^2$. This constitutes an obvious extension of the quasi-arithmetic means discussed
in subsection~\ref{sec:wqam}. The convexity of these generalized quasi-arithmetic means is characterized in Theorem 2.2 in~\cite{zhao}. Moreover,
one can characterize their concavity just by reverting the inequality in the statement of this theorem; the proof would follow identically
{\it mutatis mutandis}. As these means constitute a structural generalization of the quasi-arithmetic means, they could also be used to build
new trading functions falling under our more general umbrella, but not necessarily that
of \cite{angeris2022constant}. In such a case, one would need to consider
the functions $f_i$ to be strictly increasing as well, in order to assure the quasi-linearity of $\Phi$ (since it would become the composition
of two strictly increasing functions).

As a final note, we state the obvious fact that other generalizations of quasi-arithmetic means are indeed possible~\cite{mb}.

\section{Conclusions}
\label{sec:06}

This paper has been devoted to the mathematical analysis of a sort of AMMs, the so called constant function market makers, which exchange algorithm
is determined by a trade function. This function depends on the provision of the different liquidity pools and, keeping its constancy in every trade,
sets the asset prices. The robustness of this type of AMMs against the attacks of arbitrageurs has been established in the literature under the assumption
of a convex trade function~\cite{angeris2022constant}.
These studies formulate the problem as a mathematical optimization and give conditions on the optimality of no trading
whatsoever, what implies the absence of arbitrage, see~\cite{angeris2022constant} and references therein.

Herein, we have extended these previous works in several directions, but
always restricted to those cases in which the trading functions are
monotonically increasing, a usual assumption in the literature (which is also a characteristic property of means).
First, we have considered the case of quasilinear, but not necessarily convex, trade functions. Then, we have not only characterized the no-trade region, but actually any type of optimal trade, and the former came as a consequence of
the latter. We have also considered asset-dependent fees, rather than the uniform fee assumed in~\cite{angeris2022constant}. And finally, this all
has served us to embed the construction of new possible trade functions into the theory of generalized means, which is well-studied from the
mathematical viewpoint.

Our mathematical results aim to generalize the class of trade
functions that can be used to construct the exchange algorithm over which an AMM is designed. In this spirit, we have given families of examples of constant
function markets makers built based on trade functions that are quasilinear but not convex. Since popular AMMs such as Uniswap \cite{AZS}
and Balancer \cite{MM} use a
trade function that is a mean (be it weighted or unweighted), we have constructed our quasilinear-but-not-convex trade functions as generalized means.
The mathematical theory of generalized means is both rich and classical, what has facilitated their analysis as trade functions. Furthermore, we have
performed the numerical optimization of AMM trading for three particular cases of trade functions: the arithmetic and geometric means (both with a well-defined convexity), and another exotic mean that is quasilinear but neither concave nor convex; therefore, the former are covered by the classical theory in~\cite{angeris2022constant}, but not the latter, which is only regarded in our extension. Our numerical experiments show that all the three means behave similarly as trade functions. Therefore, quasilinear trade functions could in principle be as robust to arbitrage as convex/concave functions.

Overall, our results open the possibility of constructing new AMMs based on trading functions that are not necessarily either convex or concave, but still keep the robustness of the AMM against arbitrage attacks. Our preliminary numerical experiments suggest that some exotic means that are quasilinear but not convex (and not concave either) might serve to this purpose. Of course, more research is needed in order to check other properties of the so-generated AMMs. Clearly, we have to pay a price for the mathematical sophistication that the use of generalized means implies; but, at least in our opinion, this very same fact may carry about advantages too. In general terms, we believe that this line of research, which relies on the mathematical formalization of the AMM functioning and its systematic analysis, can serve to improve the design of constant function market makers; remarkably, it opens the possibility of doing so without the financial risk that related empirical studies might imply.
\section*{Declarations}
\begin{itemize}
\item[] \textbf{Acknowledgements} We are grateful to the reviewers for their detailed and insightful reports that helped us significantly improve the results as well as the presentation.

\item[] \textbf{Author contributions} All authors contributed equally in the preparation of the manuscript.  All authors reviewed the final version of the manuscript.

\item[]\textbf{Funding} This research has been partially supported by ANID--Chile under project Fondecyt Regular 1241040 (Lara), by MICIU / AEI / 10.13039 / 501100011033 and by ERDF, EU (Projects PID2021-125871NB-I00 (Escudero) and PID2024-156273NA-I00 (Sama)), by MICIU / AEI / 10.13039 / 501100011033 / European Union NextGenerationEU / PRTR (Project CPP2024-011557 (Escudero and Sama)), and by COST (European Cooperation in Science and Technology) through COST Action Stochastica CA24104 (Escudero).

\item[] \textbf{Data Availability:} Not applicable.

\end{itemize}

\begin{itemize}
\item[] \textbf{Conflict of interest} The authors declare that they have no conflict of interest.
\item[] \textbf{Consent for publication} Not applicable.
\item[] \textbf{Ethics approval and consent to participate} Not applicable.

\end{itemize}

\end{document}